%% file: laurent.tex
\documentclass[12pt]{amsart}

\usepackage{amsfonts}
\usepackage{amsmath}
\usepackage{amssymb}

\usepackage{mathrsfs}
\usepackage{amscd}
\usepackage[all]{xy}
\usepackage[pdftex,final]{graphicx}

\usepackage{hyperref}
\usepackage[hyperpageref]{backref}

\usepackage{exscale,txfonts}

\newtheorem{lem}{Lemma}[section]
\newtheorem{thm}[lem]{Theorem}
\newtheorem{prop}[lem]{Proposition}
\newtheorem{cor}[lem]{Corollary}
\theoremstyle{definition}

\newtheorem{exa}[lem]{Example}
\newtheorem{rem}[lem]{Remark}

\input{Laurent-command}

\title{Low-dimensional homology of $\spl{2}{\laur{k}{t}}$}
\author{Kevin Hutchinson}
\address{School of Mathematical Sciences,
 University College Dublin}
\email{kevin.hutchinson@ucd.ie}
\date{\today}

\keywords{
$K$-theory, Group Homology
}
\subjclass{19G99, 20G10}

\begin{document}
\bibliographystyle{plain}
\maketitle

\begin{abstract}
We prove 
 analogues of the fundamental theorem of $K$-theory for the second and third homology
of $\mathrm{SL}_2$ over an infinite field. The statements of the theorems involve Milnor-Witt 
$K$-theory and refined scissors congruence groups.
We use these results to calculate the low-dimensional homology 
of $\mathrm{SL}_2$ of Laurent polynomials over certain fields. 
\end{abstract}

\section{Introduction}\label{sec:intro}
Our goal in this paper is to prove unstable 
 analogues of the fundamental theorem of $K$-theory for the second and third homology
of $\mathrm{SL}_2$ over an infinite field $k$, 
and to use these results to calculate the low-dimensional homology 
of $\mathrm{SL}_2$ of Laurent polynomials over certain fields. 

We follow the approach of K. Knudson (\cite{knud:laurent}) who studied the homology of 
$\spl{2}{\laur{k}{t}}$
via the Mayer-Vietoris sequence associated its natural decomposition as an amalgamated product.

Let $k$ be an infinite field. When $n\geq 3$ then $\hoz{2}{\spl{n}{k}}\cong K_2(k)$.
Furthermore,  by \cite[Corollary 6.11]{knud:annales}  
\[\hoz{2}{\spl{n}{\laur{k}{t}}}\cong K_2(\laur{k}{t}) \mbox{ for } n\geq 3.
\]

The fundamental theorem of 
$K$-theory, combined with the fact that $K_n(A)\cong K_n(A[t])$ for regular rings $A$ 
(\cite{quillen:alg1}),  tells us that 
\[
K_2(\laur{k}{t})\cong K_2(k)\oplus K_1(k).
\]
Here the homomorphism $K_2(\laur{k}{t})\to K_1(k)$ is the composite 
\[
\xymatrix{
K_2(\laur{k}{t})\ar[r]
& K_2(k(t))\ar^-{\delta_t}[r]
&K_1(k)
}
\]
where $\delta_t$ is the tame symbol associated to the $t$-adic valuation on $k(t)$.

It follows that 
\[
\hoz{2}{\spl{n}{\laur{k}{t}}}\cong \hoz{2}{\spl{n}{k}}\oplus K_1(k)
\]
for all $n\geq 3$. 

In section \ref{sec:h2} (Theorem \ref{thm:h2}) below we prove an analogue of this for the case $n=2$:  
Let $k$ be an infinite field of characteristic not equal to $2$. 
Then
\[
\hoz{2}{\spl{2}{\laur{k}{t}}}\cong \hoz{2}{\spl{2}{k}}\oplus \mwk{1}{k}.
\]

Here $\mwk{1}{k}$ is the first Milnor-Witt $K$-theory group of the field $k$ (see section 
\ref{sec:mwk}). There is a natural surjective map $\mwk{1}{k}\to K_1(k)=k^\times$ whose kernel is 
the second power of the fundamental ideal in the Witt ring of $k$. Furthermore, for any  infinite field 
$K$ there is a natural isomorphism $\hoz{2}{\spl{2}{k}}\cong \mwk{2}{K}$ and the homomorphism 
$\hoz{2}{\spl{2}{\laur{k}{t}}}\to \mwk{1}{k}$ is the composite 
\[
\xymatrix{
\hoz{2}{\spl{2}{\laur{k}{t}}}\ar[r]
& \hoz{2}{\spl{2}{k(t)}}\cong\mwk{2}{k(t)}\ar^-{\delta_t}[r]
&\mwk{1}{k}
}
\]
where $\delta_t$ is a connecting homomorphism, analogous to the tame symbol, arising in Milnor-Witt 
$K$-theory.

The third homology of the special linear groups of fields is also closely related to $K$-theory. 
Combining \cite[Corollary 5.2]{sus:bloch} and \cite[Theorem 4.7]{hutchinson:tao2} it follows that
\[ 
\hoz{3}{\spl{n}{k}}\cong K_3(k)/(\{ -1\}\cdot K_2(k) \mbox{ for all} n\geq 3.
\]
 In particular, 
\[
\ho{3}{\spl{n}{k}}{\zhalf{\Z}}\cong K_3(k)\otimes\zhalf{\Z}
\]
for all $n\geq 3$. 

By contrast, when $n=2$, the relation between $\hoz{3}{\spl{2}{k}}$ and $K_3(k)$ is much more remote. 
There is a natural surjective map from $\hoz{3}{\spl{2}{k}}$ to $\kind{k}$ but for general infinite 
fields the kernel may be very large (see \cite[Theorem 5.1]{hut:rbl11}).  

However, recent computations by the author 
 suggest  a natural candidate for the functor
\[
\frac{\hoz{3}{\spl{2}{\laur{k}{t}}}}{\hoz{3}{\spl{2}{k}}},
\]
at least when replace $\Z$ by $\zhalf{\Z}$:
Namely, given a field $K$ with a discrete valuation $v$ and residue field $k$ there is a natural 
homomorphism $\delta_\pi:\ho{2}{\spl{2}{K}}{\zhalf{\Z}}\to \rpbker{k}\otimes\zhalf{\Z}$ 
associated to a uniformizer $\pi$. 
Here $\rpbker{k}$ is a \emph{refined scissors congruence group} (see section \ref{sec:scg} below). 
These groups arise naturally in the 
calculation of the third homology of $\mathrm{SL}_2$ of fields and local rings. 

In \cite{hut:arxivhlr}, the author has shown that for certain families of discretely valued fields 
$K$ with residue field $k$ and valuation ring $\mathcal{O}$ there is a natural exact localization     
sequence
\[
0\to \ho{3}{\spl{2}{\mathcal{O}}}{\zhalf{\Z}}\to\ho{3}{\spl{2}{K}}{\zhalf{\Z}}\to
\rpbker{k}\otimes\zhalf{\Z}\to 0.
\]

The main result of  section \ref{sec:h3} below is that for any infinite field $k$ 
there is a natural isomorphism
\[
\ho{3}{\spl{2}{\laur{k}{t}}}{\zhalf{\Z}}\cong \ho{3}{\spl{2}{k}}{\zhalf{\Z}}\oplus
\left( \rpbker{k}\otimes \zhalf{\Z}\right).
\]
Here the map $\ho{3}{\spl{2}{\laur{k}{t}}}{\zhalf{\Z}}\to \rpbker{k}\otimes \zhalf{\Z}$ is the 
composite
\[
\xymatrix{
\ho{3}{\spl{2}{\laur{k}{t}}}{\zhalf{\Z}}\ar[r]
&\ho{3}{\spl{2}{k(t)}}{\zhalf{\Z}}\ar^-{\delta_t}[r]
&\rpbker{k}\otimes\zhalf{\Z}.
}
\]

A key 
ingredient in our calculations
 is Knudson's homotopy invariance theorem (\cite[Theorem 3.4]{knud:annales}): If 
$k$ is an infinite field, then the inclusion $\spl{2}{k}\to \spl{2}{k[t]}$ induces 
an isomorphism on integral homology
\[
\hoz{\bullet}{\spl{2}{k}}\cong \hoz{\bullet}{\spl{2}{k[t]}}.
\]
However, the example of Krstic and McCool (\cite{krstic:mccool}) shows that even for $H_1$ this 
result does not extend  to more general ground rings $k$. In particular, the corresponding statement for 
the ring of polynomials in two or more variables is not true.

\subsection{Preliminaries and notation}
For a commutative ring $A$ the short exact sequence 
\[
1\to \spl{2}{A}\to\gl{2}{A}\to A^\times\to 1
\] 
gives rise to a natural action of $A^\times$ on $\hoz{\bullet}{\spl{2}{A}}$: given $a\in A^\times$,
 choose $M=M(a)\in\gl{2}{A}$ with determinant $a$. Conjugation by $M$ is an automorphism of 
$\spl{2}{A}$ 
whose induced action on $\hoz{\bullet}{\spl{2}{A}}$ depends only on $a$ and not on the choice 
of $M$. Furthermore, since $a^2$ is the determinant of a scalar matrix, the action of $a^2$ on 
$\hoz{\bullet}{\spl{2}{A}}$ is trivial for any $a\in A^\times$.  It follows that each of the groups 
$\hoz{n}{\spl{2}{A}}$ is naturally a module over the group ring $\sgr{A}$ of square classes of units of 
$A$. 

For $a\in A^\times$, the square class of $a$ will be denoted $\an{a}\in \sgr{A}$. Furthermore, the 
element $\an{a}-1$ in the augmentation ideal, $\aug{A}$, will be denoted $\pf{a}$. 

For an abelian group $M$ we will let $\zhalf{M}$ denote the $\zhalf{\Z}$-module $M\otimes\zhalf{\Z}$. 

\section{The Mayer-Vietoris sequence}
Let $k$ be an infinite field and let $G=\spl{2}{\laur{k}{t}}$. We denote the matrix
\[
\matr{t}{0}{0}{1}\in \gl{2}{\laur{k}{t}}\subset \gl{2}{k(t)}
\]
by $A(t)$. $A(t)$ normalizes $\spl{2}{k(t)}$ and $\spl{2}{\laur{k}{t}}$. We denote by $C_t$ the 
conjugation map $g\mapsto A(t)^{-1}gA(t)=g^{A(t)}$ on either of these subgroups. 
Observe, in particular, that the action of $A(t)$ by conjugation on $G$ induces on 
$\ho{\bullet}{G}{\Z}$ multiplication by $\an{t}\in \sgr{k[t,t^{-1}]}$.

Let 
$G_1$ denote the subgroup $\spl{2}{k[t]}$ of $G$ and let 
\[
G_2:= C_t(G_1)=\left\{ \matr{a}{t^{-1}b}{tc}{d}\in G\ : \ a,b,c,d\in k[t],\ ad-bc=1\right\}.
\]
Let 
\[
\Gamma =G_1\cap G_2= \left\{ \matr{a}{b}{tc}{d}\in G\ : \ a,b,c,d\in k[t],\ ad-tbc=1\right\}.
\]

K. Knudson (\cite{knud:laurent}) observes that since $\laur{k}{t}$ is a dense subring of the complete 
discretely valued field $\laurs{k}{t}$, Serre's theory of trees 
(see \cite[Chapter II, Theorem 3]{serre:trees}) allows us to deduce: 
\begin{thm}\label{thm:amalgam}
$G=\spl{2}{\laur{k}{t}}$ is the sum of the subgroups $G_1$ and $G_2$ amalgamated along their 
common intersection $\Gamma$: 
\[
G=\spl{2}{\laur{k}{t}}= \spl{2}{k[t]}\star_{\Gamma}\spl{2}{k[t]}^{A(t)}=G_1\star_{\Gamma}G_2.
\] 
\end{thm}
 
This amalgamated product decomposition gives us immediately a short exact sequence of 
$\Z[G]$-modules
\[
\xymatrix{
0\ar[r]
& \Z[G/\Gamma]\ar^-{\alpha}[r]
& \Z[G/G_1]\oplus \Z[G/G_2]\ar^-{\beta}[r]
& \Z\ar[r]
& 0.
}
\]
where $\alpha$ is the map $g\Gamma\mapsto (gG_1,gG_2)$ and $\beta$ is the unique $\Z[G]$-homomorphism
sending $(G_1,0)$ to $-1$ and $(0,G_2)$ to $1$.

The associated long exact sequence in homology -- the Mayer-Vietoris sequence of the 
amalgamated product -- takes the form
\[
\xymatrix{
\cdots\ar^-{\delta}[r]
&\hoz{i}{\Gamma}\ar^-{\alpha}[r]
&\hoz{i}{G_1}\oplus\hoz{i}{G_2}\ar^-{\beta}[r]
&\hoz{i}{\spl{2}{k[t,t^{-1}]}}\ar^-{\delta}[r]
&\cdots\\
}
\]
Implicit here are the isomorphisms of Shapiro's Lemma:
If $H$ is a subgroup of $K$, then the inclusion $\Z\to \Z[G/H]$ of $\Z[H]$-modules induces an isomorphism
$\hoz{\bullet}{H}\cong \ho{\bullet}{G}{\Z[G/H]}$.

Known results allow us to simplify further some terms in this sequence. To begin with we have the 
following homotopy-invariance property of the homology of $\spl{2}{k}$ 
(see \cite[Theorem 4.3.1]{knud:book}):

\begin{thm}
Let $k$ be an infinite field. Then the inclusion $k\to k[t]$ induces an isomorphism 
on integral homology
\[
\hoz{\bullet}{\spl{2}{k}}\cong\hoz{\bullet}{\spl{2}{k[t]}}.
\]
\end{thm}

We let $\iota:k^\times \to \spl{2}{k}$ denote the embedding 
\[
a\mapsto \matr{a}{0}{0}{a^{-1}}.
\]
The image of $\iota$ lies in the subgroup 
\[
\Gamma\cap \spl{2}{k}=B(k):=
\left\{ \matr{a}{b}{0}{a^{-1}}\ :\ a\in k^\times, b\in k\right\}.
\]
Furthermore, we have (\cite[Theorem 4.4.1]{knud:book})
\begin{thm}\label{thm:borel} 
When the field $k$ is infinite, $\iota$ induces  isomorphisms on integral 
homology
\[
\hoz{\bullet}{k^\times}\cong \hoz{\bullet}{B}\cong \hoz{\bullet}{\Gamma}.
\]
\end{thm}

With all of these identifications, together with the fact that $C_t$ induces 
an isomorphism from $G_1$ to $G_2$, the Mayer-Vietoris sequence  takes the form
\[
\xymatrix{
\cdots\ar^-{\delta}[r]
&\hoz{i}{k^\times}\ar^-{\alpha}[r]
&\hoz{i}{\spl{2}{k}}\oplus\hoz{i}{\spl{2}{k}}\ar^-{\beta}[r]
&\hoz{i}{\spl{2}{k[t,t^{-1}]}}\ar^-{\delta}[r]
&\cdots\\
}
\] 
where $\alpha(z)=(\iota(z),\iota(z))$ and $\beta(z_1,z_2)=\an{t}j(z_2)-j(z_1)$. 

As Knudson points out, the existence of this exact sequence immediately implies that 
\[
\hoz{1}{\spl{2}{\laur{k}{t}}}=0
\]
for any infinite field $k$. (This fact is originally due to P. Cohn, \cite{cohn:gl2}.)

\section{Milnor-Witt $K$-theory and $H_2$  of $SL_2$}\label{sec:mwk}
In this section all fields are of characteristic different from $2$.

The theorem of Matusmoto and Moore gives an explicit presentation of 
$\hoz{2}{\spl{2}{k}}$ for an infinite 
field $k$. In \cite{sus:tors}, A. Suslin showed that this implies a natural isomorphism
\begin{eqnarray}\label{eqn:sus}
\hoz{2}{\spl{2}{k}}\cong I^2(k)\times_{I^2(k)/I^3(k)}\milk{2}{k}
\end{eqnarray}
where $I^n(k)$ denotes the $n$-th power of the fundamental ideal of the Witt ring $W(k)$ of the field 
$k$.

We can now recognize the group on the right-hand side of this isomorphism as the second 
\emph{Milnor-Witt $K$-group} of the field $k$. 
The Milnor-Witt $K$-theory of a field is the graded ring $\mwk{\bullet}{F}$ with the following presentation 
(due to F. Morel and M. Hopkins,
see \cite{morel:trieste}):

Generators: $[a]$, $a\in F^\times$, in dimension $1$ and a further generator $\eta$ in dimension $-1$.

Relations: 
\begin{enumerate}
\item[(a)] $[ab]=[a]+[b]+\eta\cdot [a]\cdot [b]$ for all $a,b\in F^\times$
\item[(b)] $[a]\cdot[1-a]=0$ for all $a\in F^\times\setminus\{ 1\}$
\item[(c)] $\eta\cdot [a]=[a]\cdot \eta$ for all $a\in F^\times$
\item[(d)] $\eta\cdot h=0$, where $h=\eta\cdot [-1] +2\in \mwk{0}{F}$.
\end{enumerate}

Clearly there is a unique surjective homomorphism of graded rings 
$\mwk{\bullet}{F}\to\milk{\bullet}{F}$ sending $[a]$ to $\{ a\}$
 and inducing an isomorphism
\[
\frac{\mwk{\bullet}{F}}{\langle \eta \rangle}\cong \milk{\bullet}{F}
\]
(where $\milk{n}{F}:=0$ for $n<0$). 

Furthermore, there is a natural surjective homomorphism of graded 
rings $\mwk{\bullet}{F}\to I^\bullet(F)$ sending $[a]$ to $\pf{a}$ and 
$\eta$ to $\eta$. Morel shows that there is an induced isomorphism of graded rings
\[
\frac{\mwk{\bullet}{F}}{\langle h \rangle}\cong I^{\bullet}(F)
\]
(where $I^n(F):=W(F)$ for $n<0$).

The main structure theorem on Milnor-Witt $K$-theory is the following theorem of Morel:

\begin{thm}[Morel, \cite{morel:puiss}]
For a field $F$ of characteristic not $2$, the commutative square of graded rings
\begin{eqnarray*}
\xymatrix{
\mwk{\bullet}{F}\ar[r]\ar[d]
&
\milk{\bullet}{F}\ar[d]\\
I^\bullet(F)\ar[r]
&
I^{\bullet}(F)/I^{\bullet +1}(F)
}
\end{eqnarray*}
is cartesian.
\end{thm}

Thus for each $n\in \Z$ we have an isomorphism
\begin{eqnarray}\label{eqn:mwk}
\mwk{n}{F}\cong \milk{n}{F}\times_{I^n(F)/I^{n+1}(F)}I^n(F).
\end{eqnarray}

It follows that for all $n$ there is a natural short exact sequence 
\begin{eqnarray}\label{mw}
0\to I^{n+1}(F)\to \mwk{n}{F}\to \milk{n}{F}\to 0
\end{eqnarray}
where the inclusion $I^{n+1}(F)\to\mwk{n}{F}$ is given by 
\[
\pf{a_1,\ldots,a_{n+1}}\mapsto \eta[a_1]\cdots[a_{n+1}].
\]

Similarly, for $n\geq 0$, there is a short exact sequence
\[
0\to 2\milk{n}{F}\to\mwk{n}{F}\to I^n(F)\to 0
\]
where the inclusion $2\milk{n}{F}\to\mwk{n}{F}$ is given (for $n\geq 1$) by 
\[
2\{ a_1,\ldots,a_n\}\mapsto h[a_1]\cdots[a_n].
\]
For $n\geq 2$ we have
\[
h[a_1][a_2]\cdots [a_n]=([a_1][a_2]-[a_2][a_1])[a_3]\cdots [a_n]=[a^2_1][a_2]\cdots[a_n].
\] 

In particular, by (\ref{eqn:sus}) and (\ref{eqn:mwk}), there is a natural isomorphism 
\[
\hoz{2}{\spl{2}{k}}\cong \mwk{2}{k}.
\]
Indeed, the resulting diagram 
\[
\xymatrix{
\hoz{2}{\spl{2}{k}}\ar^-{\cong}[r]\ar[d]
&\mwk{2}{k}\ar[d]\\
\hoz{2}{\spl{3}{k}}\ar^-{\cong}[r] 
&\milk{2}{k}=K_2(k)\\
}
\]
commutes.

\section{The second homology of $\spl{2}{\laur{k}{t}}$}\label{sec:h2}

\textit{In this section $k$ is an infinite field of characteristic not equal to $2$.}

Our goal is to prove:
\begin{thm}\label{thm:h2} Let $k$ be an infinite field of characteristic not equal to $2$. 
Then
\[
\hoz{2}{\spl{2}{\laur{k}{t}}}\cong \hoz{2}{\spl{2}{k}}\oplus \mwk{1}{k}.
\]
\end{thm}

The Mayer-Vietoris sequence gives us an exact sequence 
\[
\xymatrix{
\hoz{2}{k^\times}\ar[r]
&\hoz{2}{\spl{2}{k}}\oplus\hoz{2}{\spl{2}{k}}\ar[r]
&\hoz{2}{\spl{2}{\laur{k}{t}}}\ar^-{\delta}[r]
&\hoz{1}{k^\times}\ar[r]
&0.\\
}
\]

However, since the homomorphism $j:\hoz{i}{\spl{2}{k}}\to\hoz{i}{\spl{2}{k[t,t^{-1}]}}$ is a 
split injection, 
we obtain an induced exact sequence 
\[
\xymatrix{
\hoz{2}{k^\times}\ar^-{\iota}[r]
&\hoz{2}{\spl{2}{k}}\ar^-{\bar{\beta}}[r]
&\dfrac{\hoz{2}{\spl{2}{\laur{k}{t}}}}{\hoz{2}{\spl{2}{k}}}\ar^-{\delta}[r]
&\hoz{1}{k^\times}\ar[r]
&0
}
\]
where $\bar{\beta}(z)=\an{t}j(z)$.

Recall that there  is a natural isomorphism  $\mwk{2}{k}\cong 
\hoz{2}{\spl{2}{k}}$. Given $a,b\in k^\times$, we will let $\hgen{a}{b}$ denote the element 
of $\hoz{2}{\spl{2}{k}}$ corresponding, under this isomorphism, to $[a][b]\in \mwk{2}{k}$.

Furthermore, there is a natural isomorphism $ \extpow{k^\times}\cong \hoz{2}{k^\times}$ 
(sending $a\wedge b$ to the homology class $([a|b]-[b|a])\otimes 1$) and, by 
the calculations of Mazzoleni (\cite{mazz:sus}), these isomorphisms fit into a commutative 
diagram
\[
\xymatrix{
\extpow{k^\times}\ar^-{\tilde{\iota}}[r]\ar^-{\cong}[d]
&\mwk{2}{k}\ar^-{\cong}[d]\\
\hoz{2}{k^\times}\ar^-{\iota}[r]
&\hoz{2}{\spl{2}{k}}
}
\]
where $\tilde{\iota}(a\wedge b)=[a][b]-[b][a]=[a^2][b]=[a][b^2]$.

Furthermore, 
Mazzoleni has shown that the image of $\tilde{\iota}$ is isomorphic to $2\cdot \milk{2}{k}\subset 
\mwk{2}{k}$ (via 
\[
\tilde{\iota}(a\wedge b)=[a^2][b] \leftrightarrow \{ a^2\}\{ b\}=2\{ a\} \{ b\}).
\] 
It follows that the cokernel of $\tilde{\iota}$ is isomorphic to $I^2(k)$, the isomorphism  being 
induced by the homomorphism 
\[
\mwk{2}{k}\to I^2(k),\quad [a][b]\mapsto \pf{a}\pf{b}.
\]

Putting all of this together, we have a natural short exact sequence
\[
\xymatrix{
0\ar[r]
&I^2(k)\ar[r]
&\dfrac{\hoz{2}{\spl{2}{\laur{k}{t}}}}{\hoz{2}{\spl{2}{k}}}\ar^-{\delta}[r]
&k^\times\ar[r]
&0.
}
\]

Morel (\cite[Theorem 2.15]{morel:algtop}) 
shows that if $F$ is a field with discrete valuation $v$, residue field $k$ 
 and corresponding uniformizer $\pi$, 
there is a well-defined residue homomorphism (depending on the choice of uniformizer) 
$\delta^v_{\pi}=\delta_{\pi}:\mwk{2}{F}\to \mwk{1}{k}$ with the properties:
\begin{enumerate}
\item $\delta_{\pi}[a][\pi]=[\bar{a}]$ whenever $v(a)=0$, and 
\item $\delta_{\pi}[a][b]=0$ if $v(a)=v(b)=0$. 
\end{enumerate}

Now let $\Delta:\hoz{2}{\spl{2}{\laur{k}{t}}}\to \mwk{1}{k}$ denote the composite
\[
\xymatrix{
\hoz{2}{\spl{2}{\laur{k}{t}}}\ar[r]
&\hoz{2}{\spl{2}{k(t)}}
\ar^-{\cong}[r]
&\mwk{2}{k(t)}\ar^-{\delta_t}[r]
&\mwk{1}{k}.
}
\]

From property (2) of $\delta_\pi$ (with $\pi =t$) it follows that the composite map 
\[
\xymatrix{\mwk{2}{k}\ar[r]
&\mwk{2}{k(t)}\ar^-{\delta_t}[r]
&\mwk{1}{k}
}
\]
is the zero map. Thus $\Delta(j(z))=0$ for all $z\in \hoz{2}{\spl{2}{k}}$ and hence $\Delta$ 
induces a well-defined homomorphsm 
\[
\tilde{\Delta}:\frac{\hoz{2}{\spl{2}{\laur{k}{t}}}}{\hoz{2}{\spl{2}{k}}}
\to \mwk{1}{k}.
\]

To complete the proof of Theorem \ref{thm:h2} we show that $\tilde{\Delta}$ is an isomorphism.  

From our work above, there is a natural injective map $I^2(k)\to \shift{\hoz{2}{\spl{2}{k}}}$ given by 
$\pf{a}\pf{b}\mapsto \an{t}j(\hgen{a}{b})$.

On the other hand, there is also a natural short exact sequence 
\[
\xymatrix{
0\ar[r]
&I^2(k)\ar[r]
&\mwk{1}{k}\ar[r]
&k^\times\ar[r]
&0
}
\]
where the inclusion $I^2(k)\to\mwk{1}{k}$ is given by $\pf{a}\pf{b}\mapsto \eta[a][b]=[ab]-[a]-[b]$
by (\ref{mw}) above.

Now, from the definition of $\tilde{\Delta}$ and the properties of $\delta_t$, we have 
\begin{eqnarray*}
\tilde{\Delta}(\an{t}j(\hgen{a}{b})&=& \delta_t(\an{t}[a][b])\\
&=& \delta_t(\eta[t][a][b]+[a][b])\\
&=& \delta_t(\eta[a][b][t])\\
&=& \eta[a][b].
\end{eqnarray*} 

We deduce immediately that there is a natural map of short exact sequences (defining the map 
$\bar{\Delta}$) 
\[
\xymatrix{
0\ar[r]
&I^2(k)\ar[r]\ar^-{=}[d]
&\dfrac{\hoz{2}{\spl{2}{\laur{k}{t}}}}{\hoz{2}{\spl{2}{k}}}\ar[r]\ar^-{\tilde{\Delta}}[d]
&k^\times\ar[r]\ar^-{\bar{\Delta}}[d]
&0\\
0\ar[r]
&I^2(k)\ar[r]
&\mwk{1}{k}\ar[r]
&k^\times\ar[r]
&0.
}
\]

We now complete the argument by showing that $\bar{\Delta}=\bar{\Delta}_k$ 
is the identity map on $k^\times$. 

We begin by observing that it is enough to show that $\bar{\Delta}_k(a)=a$ for all $a\in (k^\times)^2$ 
(and for all fields $k$), since, for any $a\in k^\times$, the  diagram 
\[
\xymatrix{
k^\times\ar[r]\ar^-{\bar{\Delta}_k}[d]
&k(\sqrt{a})^\times\ar^-{\bar{\Delta}_{k(\sqrt{a})}}[d]\\
k^\times\ar[r]
&k(\sqrt{a})^\times
}
\]
commutes, and hence $\bar{\Delta}_k(a)=\bar{\Delta}_{k(\sqrt{a})}(a)=a\in k^\times\subset k(\sqrt{a})^\times$.

Recall that $\delta$ denotes the connecting homomorphism 
$\hoz{2}{\spl{2}{\laur{k}{t}}}\to \hoz{1}{k}=k^\times$ of the Mayer-Vietoris sequence. 

For a ring $A$, we let $\tau_A$ denote the composite map 
\[
\xymatrix{
A^\times\wedge A^\times\ar^-{\cong}[r]
&\hoz{2}{A^\times}\ar^-{\iota}[r]
&\hoz{2}{\spl{2}{A}}.
}
\]

\begin{prop}\label{prop:a2}
Let $A=\laur{k}{t}$ and let $a\in k^\times$. Then $\delta(\tau_A(t\wedge a))=a^2$ in $k^\times$.
\end{prop} 
\begin{proof}
Let $B_\bullet(G)$ denote the (right) bar resolution of the group $G=\spl{2}{\laur{k}{t}}$. The 
Mayer-Vietoris sequence is derived from the long exact homology sequence of the exact sequence of 
complexes
\[
0\to B_\bullet(G)\otimes_{\Z[G]}\Z[G/\Gamma]\to 
B_\bullet(G)\otimes_{\Z[G]}(\Z[G/G_1]\oplus \Z[G/G_2])\to B_\bullet(G)\otimes_{\Z[G]}\Z
\to 0.
\] 

The homology class $\tau_A(t\wedge a)\in \ho{2}{G}{\Z}$ is represented by the cycle
\[
([\iota(t)|\iota(a)]-[\iota(a)|\iota(t)])\otimes 1\in B_2(G)\otimes_{\Z[G]}\Z.
\]

This lifts to
\[
([\iota(a)|\iota(t)]-[\iota(t)|\iota(a)])\otimes (1\cdot G_1,0) 
\in B_2(G)\otimes_{\Z[G]}( \Z[G/G_1]\oplus \Z[G/G_2]).
\]

The boundary homomorphism sends this element to 
\[
[\iota(a)]\otimes (\iota(t)\cdot G_1-1\cdot G_1,0)
\in B_1(G)\otimes_{\Z[G]}( \Z[G/G_1]\oplus \Z[G/G_2]).
\]

Now let 
\[
w=\matr{0}{-1}{1}{0}\in G_1=\spl{2}{k[t]}.
\]

Observe that 
\[
w\cdot\iota(t)= \matr{0}{-1}{1}{0}\cdot \matr{t}{0}{0}{t^{-1}}=\matr{0}{-t^{-1}}{t}{0}:=F(t)\in G_2
\]
and hence $\iota(t)=w^{-1}F(t)$ with $w^{-1}\in G_1$, $F(t)\in G_2$.

It follows that $(\iota(t)\cdot G_1-1\cdot G_1,0)$ is the image of $w^{-1}\cdot (F(t)\Gamma-\Gamma)
\in \Z[G/\Gamma]$ under the map $\Z[G/\Gamma]\to \Z[G/G_1]\oplus \Z[G/G_2]$.

Thus $\delta(\tau_A(a\wedge t))$ is represented by the cycle
\[
[\iota(a)]\otimes w^{-1}\cdot (F(t)\Gamma-\Gamma) =
\left([\iota(a)]\iota(t)-[\iota(a)]w^{-1}\right)\otimes \Gamma 
\in B_1(G)\otimes_{\Z[G]}\Z[G/\Gamma].
\]

This, in turn, is the image of 
\[
\left([\iota(a)]\iota(t)-[\iota(a)]w^{-1}\right)\otimes 1\in B_1(G)\otimes_{\Z[\Gamma]}\Z
\]
under the natural homology isomorphism 
$B_\bullet(G)\otimes_{\Z[\Gamma]}\Z\to B_{\bullet}(G)\otimes_{\Z[G]}\Z[G/\Gamma]$.

For a group $H$, we let $C_\bullet(H)$ denote the (right) homogeneous resolution of $H$. 
Then the isomorphism $B_\bullet(H)\to C_\bullet(H)$ is obtained by the map of right $\Z[H]$-complexes
\[
[h_n|\cdots|h_1]\mapsto (h_n\cdot h_{n-1}\cdots h_1,h_{n-1}\cdots h_1,\ldots,h_1,1).
\]

Thus $\left([\iota(a)]\iota(t)-[\iota(a)]w^{-1}\right)\otimes 1\in B_1(G)\otimes_{\Z[\Gamma]}\Z$ 
corresponds to 
\[
\left( (\iota(at),\iota(t))-(\iota(a)w^{-1},w^{-1})\right)\otimes 1\in \C_1(G)\otimes_{\Z[\Gamma]}\Z. 
\]

To construct an augmentation preserving map of $\Z[\Gamma]$-resolutions of $\Z$ from 
$C_\bullet(G)$ to $C_\bullet(\Gamma)$ we choose a set-theoretic section $s:G/\Gamma\to G$ and 
send $(g_n,\ldots,g_0)$ to $(\overline{g_n},\ldots,\overline{g_0})$, where 
$\overline{g}:=s(g\Gamma)^{-1}g\in \Gamma$. 

To be more specific, we choose a section $s$ satisfying $s(\iota(a)w^{-1}\Gamma)=w^{-1}$  
and $s(\iota(at)\Gamma)=\iota(t)$ for all 
$a\in k^\times$.

Thus the element   $(\iota(at),\iota(t))-(\iota(a)w^{-1},w^{-1})\in C_1(G)$ maps to 
$(\iota(a),1)-(\iota(a^{-1}),1)\in C_1(\Gamma)$ (since $w\iota(a)w^{-1}=\iota(a^{-1})$ in $G$).

Finally the homology class 
\[
\left( (\iota(a),1)-(\iota(a^{-1}),1\right )\otimes 1\in C_1(\Gamma)\otimes_{\Z[\Gamma]}\Z
\]
corresponds to $\iota(a)\cdot \iota(a^{-1})^{-1}=\iota(a^2)$ under the isomorphism 
$\hoz{1}{\Gamma}\cong \ab{\Gamma}$ and hence to $a^2\in k^\times\cong\ab{\Gamma}$.  
\end{proof}

\begin{cor} For any field $k$, we have $\bar{\Delta}_k(a^2)=a^2$ for all $a\in k^\times$.  
\end{cor}
\begin{proof} By Proposition \ref{prop:a2}, $a^2\in k^\times$ is the image, under $\delta$, of 
the element of $\shift{\hoz{2}{\spl{2}{k}}}$ represented by $\tau_A(a\wedge t)$. Therefore,
it is enough to verify that $\Delta(\tau_A(a\wedge t))=[a^2]\in\mwk{1}{k}$:

The image of $\tau_A(a\wedge t)$ in $\mwk{2}{k(t)}$ is $\tilde{\iota}(a\wedge t)= [a^2][t]$. 

Hence, from the definitions,
\[
\bar{\Delta}_k(a^2)=\Delta(\tau_A(a\wedge t))=\delta_t([a^2][t])=[a^2]
\] 
as required.
\end{proof}

\section{Some Examples and Special Cases}

The main result of section \ref{sec:h2} is that if $k$ is an infinite field of characteristic not 
equal to $2$, then 
\[
\hoz{2}{\spl{2}{\laur{k}{t}}}\cong \hoz{2}{\spl{2}{k}}\oplus \mwk{1}{k}\cong \mwk{2}{k}\oplus\mwk{1}{k}.
\]

This allows us to calculate these homology groups for certain families of fields for which the 
Milnor $K$-theory and Witt rings are known. 

For a global field $F$, we will let $\Omega=\Omega(F)$ be the set of real embeddings of $F$. 
We denote by $K_2(F)_+$ the kernel of the (split) surjection 
 $K_2(F)\to\mu_2^{\Omega}$ induced by the Hilbert 
symbols associated to each of the real embeddings.

\begin{prop}
Let $F$ be a global field. Then there is a natural split exact sequence 
\[
0\to K_2(F)_+ \to \mwk{2}{F}\to \Z^{\Omega}\to 0. 
\]
\end{prop}
\begin{proof} We begin by recalling from quadratic form theory 
that for all $n\geq 1$, $I^n(\R)\cong \Z$ with 
generator $\pf{-1}\cdots \pf{-1}$.  

For any real embedding $\sigma:F\to\R$, let $T_\sigma$ denote the composite homomorphism
\[
\mwk{2}{F}\to\mwk{2}{\R}\to I^2(\R)\cong \Z.
\]
Thus $T_\sigma([a][b])=\pf{\sgn{\sigma(a)}}\pf{\sgn{\sigma(b)}}$ where
\[
\sgn{x}=
\left\{
\begin{array}{ll}
1,&x>0\\
-1,&x<0.\\
\end{array}
\right.
\] 

Let 
\[
T=\oplus_{\sigma\in \Omega}T_\sigma:\mwk{2}{F}\to I^2(\R)^{\Omega}
\]
and let $\mwk{2}{F}_+:=\ker{T}$. 

Thus we obtain a commutative diagram with exact rows and columns
\[
\xymatrix{
&&0\ar[d]&0\ar[d]&\\
&&I^3(F)\ar^-{T}[r]\ar[d]&I^3(\R)^{\Omega}\ar[d]&\\
0\ar[r]
&\mwk{2}{F}_+\ar[r]\ar[d]
&\mwk{2}{F}\ar[d]\ar^-{T}[r]
&I^2(\R)^{\Omega}\ar[r]\ar[d]
&0\\
0\ar[r]
&K_2(F)_+\ar[r]
&K_2(F)\ar[d]\ar[r]
&\mu_2^{\Omega}\ar[r]\ar[d]
&0\\
&&0&0&\\
}
\]
where $I^2(\R)\to \mu_2$ is the unique homomorphism sending $\pf{-1}\pf{-1}$ to $-1$. 

Now quadratic form theory tells us that, for global fields $F$, the map 
$T:I^3(F)\to I^3(\R)^{\Omega}$ is an isomorphism.  The result follows immediately.
\end{proof}

Since $K_2(\Q)_+\cong \bigoplus_{p\mbox{ \tiny  odd}}\F{p}^\times$ we immediately deduce:
\begin{cor}
$\hoz{2}{\spl{2}{\Q}}\cong \mwk{2}{\Q}\cong \Z\oplus\left( \bigoplus_{p\mbox{ \tiny odd}}\F{p}^\times\right)$. 
\end{cor}

The situation for $\mwk{1}{F}$ is slightly more complicated:

Again, we set 
\[
\mwk{1}{F}_+=\ker{\mwk{1}{F}\to I(\R)^{\Omega}} 
\]
and
\[
K_1(F)_+= F^\times_+=\ker{F^\times \to \mu_2^{\Omega}}
\]
(the map being induced by the sgn homomorphisms).

Let $k_2(F):=K_2(F)\otimes \Z/2$ and let $k_2(F)_+:=\ker{k_2(F)\to \mu_2^{\Omega}}$.
\begin{prop}
Let $F$ be a global field. There is a split short exact sequences
\[
0\to \mwk{1}{F}_+\to \mwk{1}{F}\to I(\R)^{\Omega}\to 0
\]
and a short exact sequence 
\[
0\to k_2(F)_+\to \mwk{1}{F}_+\to F^\times_+\to 0
\]
which is split if $\Omega\not=\emptyset$.

In particular, if $F$ is a number field admitting a real embedding, as an additive group,
\[
\mwk{1}{F}\cong \Z^{\Omega}\oplus k_2(F)_+\oplus F^\times_+. 
\]
\end{prop}
\begin{proof}
Let $\mathcal{T}:=\ker{I^2(F)\to I^2(\R)^{\Omega}}$. Then there is a commutative diagram with 
exact rows and columns
\[
\xymatrix{
&&0\ar[d]&0\ar[d]&\\
&&I^3(F)\ar^-{T}_-{\cong}[r]\ar[d]&I^3(\R)^{\Omega}\ar[d]&\\
0\ar[r]
&\mathcal{T}\ar[r]\ar[d]
&I^2(F)\ar[d]\ar[r]
&I^2(\R)^{\Omega}\ar[r]\ar[d]
&0\\
0\ar[r]
&k_2(F)_+\ar[r]
&k_2(F)\ar[d]\ar[r]
&\mu_2^{\Omega}\ar[r]\ar[d]
&0\\
&&0&0&\\
}
\]
It follows that $\mathcal{T}\cong k_2(F)_+$.

Therefore, we obtain a commutative diagram with exact rows and columns
\[
\xymatrix{
&0\ar[d]&0\ar[d]&0\ar[d]&\\
0\ar[r]&
k_2(F)_+\ar[r]\ar[d]&
I^2(F)\ar[r]\ar[d]&I^2(\R)^{\Omega}\ar[d]&\\
0\ar[r]
&\mwk{1}{F}_+\ar[r]\ar[d]
&\mwk{1}{F}\ar[d]\ar[r]
&I(\R)^{\Omega}\ar[r]\ar[d]
&0\\
0\ar[r]
&F^\times_+\ar[r]
&K^\times\ar[d]\ar[r]
&\mu_2^{\Omega}\ar[r]\ar[d]
&0\\
&&0&0&\\
}
\]

Finally, if $\Omega\not=\emptyset$, then $F^\times_+$ is a free abelian group and hence the sequence 
\[
0\to k_2(F)_+\to \mwk{1}{F}_+\to F^\times_+\to 0
\]
is split.
\end{proof}

Since $\Q^\times_+\cong \bigoplus_p \Z$ and $k_2(\Q)_+\cong \bigoplus_{p\mbox{ \tiny odd}}\Z/2$ we 
deduce:
\begin{cor}
$\mwk{1}{\Q}\cong \Z\oplus\left( \bigoplus_p \Z\right)\oplus \left(\bigoplus_{p\mbox{ \tiny odd}}\Z/2\right)$.
\end{cor}

Putting this together with the result of section \ref{sec:h2} gives:
\begin{thm}
$\hoz{2}{\spl{2}{\laur{\Q}{t}}}\cong \Z\oplus\left( \bigoplus_p \Z\right)\oplus
\left(\bigoplus_{p\mbox{ \tiny odd}}(\F{p}^\times\oplus \Z/2)\right)$. 
\end{thm}

The results of section \ref{sec:h2} also tell us about the stabilization homomorphism
from $\hoz{2}{\spl{2}{\laur{k}{t}}}$ to $\hoz{2}{\mathrm{SL}(\laur{k}{t})}$:

\begin{prop} Let $k$ be an infinite field of characteristic not equal to $2$.
Then the natural stabilization homomorphism 
\[
\hoz{2}{\spl{2}{\laur{k}{t}}}\to \hoz{2}{\mathrm{SL}(\laur{k}{t})}= K_2(\laur{k}{t})
\]
is surjective with kernel isomorphic to $I^3(k)\oplus I^2(k)$. 
\end{prop}

\begin{proof} 
By the Fundamental Theorem of Algebraic $K$-theory (see \cite[V.6]{weibel:kbook}) there is a natural 
split exact sequence 
\[
0\to K_2(k)\to K_2(\laur{k}{t})\to K_1(k)\to 0.
\]

Let $\mathcal{K}$ denote the kernel of the map
\[
\hoz{2}{\spl{2}{\laur{k}{t}}}\to \hoz{2}{\mathrm{SL}(\laur{k}{t})}.
\]

Thus we get a natural commutative diagram whose rows are split exact (with compatible 
splitting maps):
\[
\xymatrix{
&0\ar[d]
&0\ar[d]
&0\ar[d]
&\\
0\ar[r]
&I^3(k)\ar[r]\ar[d]
&\mathcal{K}\ar[r]\ar[d]
&I^2(k)\ar[r]\ar[d]
&0\\
0\ar[r]
&\mwk{2}{k}\ar[r]\ar[d]
&\hoz{2}{\spl{2}{\laur{k}{t}}}\ar[r]\ar[d]
&\mwk{1}{k}\ar[r]\ar[d]
&0\\
0\ar[r]
&K_2(k)\ar[r]\ar[d]
&\hoz{2}{\mathrm{SL}(\laur{k}{t})}\ar[r]\ar[d]
&K_1(k)\ar[r]\ar[d]
&0\\
&0
&0
&0
&\\
}
\]

\end{proof}

\begin{cor}
\begin{enumerate}
\item 
Let $k$ be a quadratically closed  field of characteristic not equal to $2$.
Then the natural stabilization homomorphism 
\[
\hoz{2}{\spl{2}{\laur{k}{t}}}\to \hoz{2}{\mathrm{SL}(\laur{k}{t})}
\]
is an isomorphism.
\item
Let $k$ be an infinite field of characteristic not equal to $2$ such that $I^3(k)=0$. Then there is a 
natural short exact sequence
\[
0\to k_2(k)\to \hoz{2}{\spl{2}{\laur{k}{t}}}\to \hoz{2}{\mathrm{SL}(\laur{k}{t})}
\to 0.
\]
\end{enumerate}
\end{cor}

\begin{proof}
\begin{enumerate}
\item When $k$ is quadratically closed $I^n(k)=0$ for all $n\geq 1$. 
\item 
In this case
\[
I^2(k)=I^2(k)/I^3(k)\cong k_2(k)
\]
by Merkurjev's theorem.
\end{enumerate}
\end{proof}
\begin{rem} The condition $I^3(k)=0$ is satisfied, for example, global fields of positive 
characteristic and by totally imaginary 
number fields $k$ (and for these fields $K_2(k)\otimes \Z/2Z$ is infinite). More generally, this condition 
is satisfied by fields of cohomological dimension $2$ or less.  
\end{rem}

\section{Scissors Congruence Groups} \label{sec:scg}
\subsection{The classical scissors congruence group}
Let $k$ be a field with at least four elements. The \emph{scissors congruence group}, $\pb{k}$, 
 of $k$ is the $\Z$-module with generators $\gpb{a}$, $a\in k^\times$, subject to the relations
$\gpb{1}=0$ and 
\[
\gpb{x}-\gpb{y}+\gpb{\frac{y}{x}}-\gpb{\frac{1-x^{-1}}{1-y^{-1}}}+\gpb{\frac{1-x}{1-y}},\ x,y\not=1.
\]  

\begin{rem}
When $k=\C$, the minus-eigenspace of $\pb{\C}$ for the action of complex conjugation can be naturally 
identified with the scissors congruence group of polyhedra in $3$-dimensional hyperbolic space 
(\cite{sah:dupont}). This is the origin of the name. 
\end{rem}

For finite fields we have the following calculation (\cite[Lemma 7.4]{hut:cplx13}):
\begin{thm}\label{thm:pbfinite}
 Let $n'$ denote the odd part of the integer $n$. Let $\F{q}$ denote the finite field 
with $q$ elements. Then 
\[
\zhalf{\pb{\F{q}}}\cong \Z/(q+1)'.
\]
\end{thm}

We let 
\[
\asym{2}{\Z}{k^\times}:=
\frac{k^\times\otimes_\Z k^\times}{\langle \{ x\otimes y+y\otimes x| x,y\in k^\times\}\rangle},
\]
the second (graded) symmetric power. 
We let $x \asymm y$ denote the image of $x\otimes y$ in $\asym{2}{\Z}{k^\times}$. 

The \emph{Bloch group}, $\bl{k}$, of the field $k$ is the kernel of the map
\[
\lambda: \pb{k}\to \asym{2}{\Z}{k^\times}, \ \gpb{a}\mapsto a\asymm (1-a).
\]

By a result of Suslin (\cite{sus:bloch}), 
the Bloch group of a field is very closely related to the indecomposable $K_3$ of the field:

\begin{thm} For any field $k$ there is a natural short exact sequence 
\[
0\to \widetilde{\mathrm{tor}(\mu_k,\mu_k)}\to \kind{k}\to \bl{k}\to 0.
\] 
\end{thm}

Here $\widetilde{\mathrm{tor}(\mu_k,\mu_k)}={\mathrm{tor}(\mu_k,\mu_k)}$ when $\mathrm{char}(k)=2$ and 
otherwise is the nontrivial extension of ${\mathrm{tor}(\mu_k,\mu_k)}$ by $\Z/2$. 

\subsection{The refined scissors congruence group}
The \emph{refined scissors congruence group}, $\rpb{k}$, 
 of $k$ is the $\sgr{k}$-module with generators $\gpb{a}$, $a\in k^\times$, subject to the relations
$\gpb{1}=0$ and 
\[
\gpb{x}-\gpb{y}+\il{x}\gpb{\frac{y}{x}}-\il{x^{-1}-1}\gpb{\frac{1-x^{-1}}{1-y^{-1}}}+
\il{1-x}\gpb{\frac{1-x}{1-y}},\ x,y\not=1.
\]   

We endow 
$\asym{2}{\Z}{k^\times}$  with the trivial $\sgr{k}$-module structure. 

Aa in \cite{hut:cplx13},  we let $\lambda_1:\rpb{k}\to \sgr{k}$ be the $\sgr{k}$-module homomorphism 
 sending $\gpb{a}$ to $\pf{a}\pf{1-a}$ and let $\lambda_2:\rpb{k}\to \asym{2}{\Z}{k^\times}$ be the 
$\sgr{k}$-homomorphism sending $\gpb{a}$ to $a\asymm (1-a)$.  

Furthermore, let $\Lambda:=(\lambda_1,\lambda_2):\rpb{k}\to \sgr{k}\oplus \asym{2}{\Z}{k^\times}$. 

The \emph{refined Bloch group} of $k$, $\rbl{k}$, is the kernel of $\Lambda$. 

It has the same relation to $\hoz{3}{\spl{2}{k}}$ as $\bl{k}$ does to $\kind{k}$:

\begin{thm}\label{thm:sl2rbl}(\cite[Theorem 4.3]{hut:cplx13})
Let $k$ be a field with at least $28$ elements. There is a natural complex
\[
0\to {\mathrm{tor}(\mu_k,\mu_k)}\to \hoz{3}{\spl{2}{k}}\to \rbl{k}\to 0
\]
which is exact except possibly at the middle term where the homology is annihilated by $4$.
\end{thm}

The natural map $\rpb{k}\to \pb{k}$ induces a homomorphism $\rbl{k}\to \bl{k}$.
We denote the kernel of this map by $\rblker{k}$. 

Here we collect some of the relevant facts about $\rblker{k}$: 

\begin{thm}\label{thm:rblker} Let  $k$ be a field with at least $4$ elements.
\begin{enumerate}
\item The map $\rbl{k}\to \bl{k}$ is surjective. Hence there is a natural short exact sequence 
of $\sgr{k}$-modules 
\[
0\to \rblker{k}\to \rbl{k}\to \bl{k}\to 0. 
\] 

Furthermore, if $k$ has finitely many square classes, after tensoring with $\zhalf{\Z}$ this sequence 
is split.
\item If $k$ has at least $28$ elements there is a natural short exact sequence of $\sgr{k}$-modules
\[
0\to \zhalf{\rblker{k}}\to \ho{3}{\spl{2}{k}}{\zhalf{\Z}}\to \zhalf{\kind{k}}\to 0.
\]
Furthermore, if $k$ has finitely many square classes, this sequence 
is split.
\item For any infinite field $k$ 
there is a natural exact sequence
\[
0\to \zhalf{\rblker{k}}\to \ho{3}{\spl{2}{k}}{\zhalf{\Z}}\to \ho{3}{\spl{n}{k}}{\zhalf{\Z}}\to 
\zhalf{\milk{3}{k}}\to 0
\] 
for any $n\geq 3$.
\item If $k$ is finite or real-closed or quadratically closed, then $\rblker{k}=0$. 
\item Suppose that $k$ is a local field with finite residue field $\bar{k}$ 
of odd order. If $\Q_3\subset k$ 
we suppose that $[k:\Q_3]$ is odd. Then there is a natural isomorphism
\[
\zhalf{\rblker{k}}\cong \zhalf{\pb{\bar{k}}}. 
\]
\item Suppose that $k$ is the field of fractions of a unique factorization domain $A$. Let 
$\mathcal{P}$ be a set of representatives of the association classes of prime elements of $\mathcal{P}$. 
The there is a natural surjective homomorphism 
\[
\xymatrix{
\zhalf{\rblker{k}}\ar@{>>}[r]
&\bigoplus_{p\in \mathcal{P}}\zhalf{\pb{\bar{k}_p}}.
}
\] 
\end{enumerate}
\end{thm}
\begin{proof}
\begin{enumerate}
\item This is \cite[Corollary 5.1]{hut:cplx13}.
\item This is \cite[Lemma 5.2]{hut:cplx13}.

The statement about the splitting of the sequences follows from the fact that 
$\zhalf{\rblker{k}}=\aug{k}\zhalf{\rbl{k}}$, together with the fact that if $k^\times/(k^\times)^2$ is finite
and if $M$ is any $\zhalf{\sgr{k}}$-module, then the sequence 
\[
0\to \aug{k}M\to M\to M/\aug{k}M\to 0
\]  
is naturally split.
\item This follows from \cite[Theorem 4.7]{hutchinson:tao2} and \cite[Lemma 5.2]{hut:cplx13}. 
\item For quadratically closed field the is immediate from the definition. For real closed fields this 
is the result of Parry and Sah \cite{sah:parry}. For finite fields, this is 
\cite[Lemma 7.1]{hut:cplx13}.
\item This is \cite[Theorem 6.19]{hut:rbl11}.
\item This is \cite[Theorem 5.1]{hut:rbl11}.
\end{enumerate}
\end{proof}
\subsection{The module $\rpbker{k}$}

The $\sgr{k}$-module 
\[
\rpbker{k}:=\ker{\lambda_1:\rpb{k}\to\asym{2}{\Z}{k^\times}}
\]
 will play a key role in our calculations below. 

Note that $\rbl{k}$ is the kernel of $\lambda_2:\rpbker{k}\to \asym{2}{\Z}{k^\times}$. 

\begin{lem}\label{lem:rpbker}
There is a natural short exact sequence of $\sgr{k}$-modules
\[
0\to \zhalf{\rblker{k}}\to\zhalf{\rpbker{k}}\to \zhalf{\pb{k}}\to 0 
\]
which is split if $k$ has finitely many square classes.
\end{lem}

\begin{proof}
For $x\in k^\times$ we let $\suss{1}{x}$ denote the element $\gpb{x}+\an{-1}\gpb{x^{-1}}\in \rpb{k}$. 

An elementary calculation (see, for example, \cite[Lemma 3.3]{hut:rbl11}) shows that for any 
$1\not= x\in k^\times$ 
\[
r(x):= (\an{-1}+1)\gpb{x}+\pf{1-x}\suss{1}{x}\in \ker{\lambda_1}=\rpbker{k}. 
\]

But $\lambda_2(r(x))= 2\lambda(\gpb{x})$. 

Let
\[
\mathcal{R}(k):=\image{\lambda:\pb{k}\to \asym{2}{\Z}{k^\times}}\mbox{ and }
\mathcal{R}_1(k):=\image{\lambda_2:\rpbker{k}\to \asym{2}{\Z}{k^\times}}.
\]

Then $\mathcal{R}_1(k)\subset \mathcal{R}(k)$ and the quotient, $C(k)$, is annihilated by $2$. 

The statement now follows by applying the snake lemma to the map of short exact sequences
\[
\xymatrix{
0\ar[r]
&\rbl{k}\ar[r]\ar[d]
&\rpbker{k}\ar[r]\ar[d]
&\mathcal{R}_1(k)\ar[r]\ar[d]
&0\\
0\ar[r]
&\bl{k}\ar[r]
&\pb{k}\ar[r]
&\mathcal{R}(k)\ar[r]
&0\\
}
\]
and tensoring with $\zhalf{\Z}$.
\end{proof}

We let $\ks{1}{k}$ denote the $\sgr{k}$-submodule of $\rpb{k}$ generated by the elements 
$\suss{1}{x}$, $x\in k^\times$. 

Let $\qrpb{k}$ denote the $\sgr{k}$-module $\rpb{k}/\ks{1}{k}$.
Let $\qrpbker{k}$ denote the image of the map $\rpbker{k}\to\qrpb{k}$. 

In \cite[Lemma 3.3]{hut:rbl11} it is shown that $\rpbker{k}\cap \ks{1}{k}$ is 
precisely the torsion subgroup of $\ks{1}{k}$ and is annihilated by $4$. We deduce:

\begin{lem} The natural map $\rpbker{k}\to \qrpbker{k}$ induces an isomorphism
\[
\zhalf{\rpbker{k}}\cong\zhalf{\qrpbker{k}}.
\]
\end{lem}

\subsection{Specialization homomorphisms}
 
Suppose given  a field $F$ with valuation $v:F^\times\to\Gamma$.

Let $O_v$ be the corresponding valuation ring, let $U_v$ the group of units of $O_v$ and 
let $k$ be the residue field. Given $a\in U_v$, we denote by $\bar{a}$ its image in $k$. 

Given an $\sgr{k}$-module $M$, we let $M_F$ denote the $\sgr{F}$-module 
\[
\sgr{F}\otimes_{\Z[U_v/U_v^2]}M.
\]

The following is \cite[Theorem 4.9]{hut:rbl11} :

\begin{thm}
 There is a natural $\sgr{F}$-module homomorphism  $S_v:\rpb{F}\to \qrpb{k}_F$  which sends 
$\gpb{a}$ to $1\otimes \gpb{\bar{a}}$ when $a\in U_v$.
\end{thm}

It can be easily verified that $S_v$  restricts to a map 
$S_v:\qrpbker{F}\to \qrpbker{k}_F$  (see \cite[Section 5.2]{hut:arxivhlr}).

When $v$ is a discrete value and $\pi$ is a uniformizer, then there is a $\Z[U_v/U_v^2]$-module 
decomposition $\sgr{F}\cong \Z[U_v/U_v^2]\oplus \an{\pi}\cdot\Z[U_v/U_v^2]$ where 
$\an{\pi}\in  F^\times/(F^\times)^2$ is the square class of $\pi$. Thus for any $\sgr{k}$-module $M$, 
there is a $\Z[U_v/U_v^2]$-decomposition (depending on the choice of $\pi$)
\[
M_F\cong M\oplus \an{\pi}\cdot M.
\] 
We let $\rho_\pi:M_F\to M$ denote the resulting $\Z[U_v/U_v^2]$-map 
\[
M_F\to \an{\pi}\cdot M \cong M
\]  
arising from projection on the second factor.

When $v$ is a discrete value and $\pi$ is a uniformizer, we let $\delta_\pi$ denote the composite
\[
\xymatrix{
&\rpbker{F}\ar^-{S_v}[r]
&\qrpbker{k}_F\ar^-{\rho_\pi}[r]
&\qrpbker{k}.
}
\]

\section{The third homology of $\spl{2}{\laur{k}{t}}$}\label{sec:h3}
In this section we prove the following result:

\begin{thm}\label{thm:h3}
Let $k$ be an infinite field. Then there is a natural isomorphism
\[
\ho{3}{\spl{2}{\laur{k}{t}}}{\zhalf{\Z}}\cong \ho{3}{\spl{2}{k}}{\zhalf{\Z}}\oplus 
\zhalf{\rpbker{k}}.
\]
\end{thm}

\subsection{Preliminaries} Let $G$ be a group and $L_\bullet$ a nonnegative 
acyclic complex of $\Z[G]$-modules augmented over $\Z$ via a map 
$\epsilon: L_0 \to \Z$; i.e. $\epsilon$ induces a weak equivalence of 
complexes of $\Z[G]$-modules $L_\bullet \simeq \Z[0]$.  

If we suppose further that each of the $\Z$-modules $L_n$ is free then 
for any abelian group $M$, $\epsilon\otimes \id{M}$ induces 
 a weak equivalence of complexes 
\[
L_\bullet\otimes_{\Z} M \simeq \Z[0]\otimes_{\Z}M=M[0].
\]
If $M$ is furthermore a $\Z[G]$-module, then this is a weak equivalence of 
complexes of $\Z[G]$-modules if $G$ acts diagonally on $L_\bullet\otimes_{\Z} M $.
It follows that there is an induced isomorphism of homology groups
\[
\ho{n}{G}{L_\bullet\otimes_{\Z} M }\cong \ho{n}{G}{M}.
\]
Here the left-hand term is the hyperhomology of $G$ with coefficients in the 
complex $L_\bullet\otimes_{\Z} M $. This, by definition, is the homology 
of the total complex $T_\bullet(G,M)$ associated to the double complex 
$D_{p,q}(G,M)=B_p(G)\otimes_{\Z[G]}(L_q\otimes_{\Z}M)$. 

Associated to the double complex $D_{p,q}(G,M)$ is a spectral sequence of the form
\[
E^1_{p,q}(G,M)=H_p(G,L_q\otimes M) \Rightarrow H_{p+q}(G,M). 
\]

We make the following observations about this construction:
\begin{enumerate}
\item Functoriality: Given a map of pairs $(H,M')\to (G,M)$, and endowing $L_\bullet$ with the structure
of a  $\Z[H]$-complex via the map $H\to G$, we obtain a natural maps of complexes 
$D_{p,q}(H,M')\to D_{p,q}(G,M)$ and $T_\bullet(H,M')\to T_\bullet(G,M)$. 

\item 
 The functor $M\mapsto L_\bullet\otimes M$ from $\Z[G]$ modules to $Z[G]$-complexes 
is exact since the $L_q$ are $\Z$-free, and hence so are the functors
 $M\mapsto \{ D_{p,q}(G,M)\}_{p,q}$  and $M\mapsto T_\bullet(G,M)$. 

\item The natural map  $D_{\bullet,0}(G,M)\to T_\bullet(G,M)$  is a map of complexes for which
the resulting edge homomorphism
 on homology $\ho{\bullet}{G}{L_0\otimes M}\to \ho{\bullet}{G}{M}$ coincides with 
that induced by $\epsilon\otimes\id{M}$.  

\item Given subgroups $H\subset K \subset G$, there are natural composite 
maps of double complexes
\[
D_{p,q}(H,\Z) \to D_{p,q}(H, \Z[G/K])\to D_{p,q}(G, \Z[G/K]).
\]

Here, the first map is induced from the inclusion of $\Z[H]$-modules 
\[
\Z\to \Z[G/K], 1\mapsto K.
\]
Thus it sends an element of the form $z\otimes \ell\in D_{p,q}(H,\Z)=B_p(H)\otimes L_q$ to 
$z\otimes (\ell\otimes K)\in D_{p,q}(H, \Z[G/K])= B_p(H)\otimes (L_q\otimes \Z[G/k])$. 
For convenience, we will use the following notation: If $w\in D_{p,q}(H,\Z)$ we will let 
$w\otimes K$ denote it's image in $D_{p,q}(H, \Z[G/K])$ or $D_{p,q}(G, \Z[G/K])$. 

\end{enumerate}

In this article, the relevant example is the case where $G=\spl{2}{k}$ and $L_q$ is the free 
abelian group of $(q+1)$-tuples of distinct elements of $\projl{k}$ (with the usual action 
of $\spl{2}{k}$ on $\projl{k}$). The boundary map $d_q:L_{q+1}\to L_q$ is the standard simplicial 
boundary map.

We will require the following facts about the associated spectral sequence
\[
E^1_{p,q}:=E^1_{p,q}(\spl{2}{k},{\Z})\Rightarrow \ho{p+q}{\spl{2}{k}}{{\Z}}:
\]

\begin{lem}\label{lem:ss1}
The map 
\begin{eqnarray*}
B_p(k^\times)\otimes_{\Z[k^\times]}{\Z}&\to & B_p(\spl{2}{k})\otimes_{\Z[\spl{2}{k}]}{L_0}= D_{p,0}\\
z\otimes 1& \mapsto & z\otimes (\infty)\\
\end{eqnarray*}
induces an isomorphism
\[
\ho{p}{k^\times}{{\Z}}\cong E^1_{p,0}=\ho{p}{\spl{2}{k}}{{L_0}}.
\]
With this identification, the edge homomorphism 
\[
\xymatrix{
\ho{p}{k^\times}{{\Z}}\cong E^1_{p,0}\ar@{->>}[r]
&E^\infty_{p,0}\ar@{^{(}->}[r]
&\ho{p}{\spl{2}{k}}{{\Z}}
}
\]
is identified with the map $\iota$. In particular, $E^\infty_{p,0}\cong 
\ho{p}{k^\times}{{\Z}}/\ker{\iota}$.
\end{lem}

\begin{proof}
$L_0$ is  a transitive permutation module over $\spl{2}{k}$ and the stabilizer of $(\infty )$ is the 
subgroup $B=B(k)$ of upper-triangular matrices. By Shapiro's Lemma it follows that the map of complexes
\[
B_\bullet(B)\otimes_{\Z[B]}\Z \to D_{\bullet,0},\quad z\otimes 1\mapsto z\otimes (\infty )
\]
induces an isomorphism on homology $\ho{\bullet}{B}{\Z}\cong\ho{\bullet}{\spl{2}{k}}{L_0}$.

However, for any infinite field $k$, the natural inclusion $k^\times \to B$ induces an isomorphism 
on homology $\ho{\bullet}{k^\times}{\Z}\cong\ho{\bullet}{B}{\Z}$ by Theorem \ref{thm:borel} above. 
\end{proof}
\begin{lem}\label{lem:ss2}
\begin{enumerate}
\item $2\cdot E^\infty_{2,1}=0=2\cdot E^\infty_{1,2}$ and hence 
$\zhalf{E^\infty_{2,1}}=0=\zhalf{E^\infty_{1,2}}$
\item $\zhalf{E^3_{0,3}}\cong\zhalf{\rpbker{k}}$.
\item 
$\zhalf{E^3_{2,0}}\cong \ho{2}{k^\times}{\zhalf{\Z}}\cong \zhalf{(k^\times \wedge k^\times)}$ and if  we 
let $\rho $ denote the isomorphism 
\[
\zhalf{(k^\times\wedge k^\times)}\to \zhalf{\asym{2}{\Z}{k^\times}},\ a\wedge b \mapsto 2(a\asymm b)
\]
then we have an equality of maps 
\[
\rho\circ d^3=\lambda_2:\zhalf{\rpbker{k}}=E^3_{3,0}\to \zhalf{\asym{2}{\Z}{k^\times}}.
\]
\end{enumerate}
\end{lem}

The proofs of these facts can be found in section 4 of \cite{hut:cplx13} (where low-dimensional 
 terms 
of the spectral sequence associated to $D_{p,q}(\spl{2}{k},\Z)$ are calculated). 

In particular, we have the following:
\begin{lem} \label{lem:h3}
$\zhalf{E^\infty_{3,0}}\cong\zhalf{\rbl{k}}$ and the inclusion $k^\times\to\spl{2}{k}$ gives 
rise to a natural exact sequence 
\[
\ho{3}{k^\times}{\zhalf{\Z}}\to \ho{3}{\spl{2}{k}}{\zhalf{\Z}}\to \zhalf{\rbl{k}}\to 0.
\]
\end{lem}

\subsection{The Mayer-Vietoris sequence again}
When $i=3$ in the Mayer-Vietoris sequence, we obtain an exact sequence 
\[
\xymatrix{
\hoz{3}{k^\times}\ar[r]
&\hoz{3}{\spl{2}{k}}\oplus\hoz{2}{\spl{2}{k}}\ar[r]
&\hoz{2}{\spl{2}{\laur{k}{t}}}\ar^-{\delta}[r]
&\hoz{2}{k^\times}\ar[r]
&\cdots\\
}
\]

Thus, tensoring with $\zhalf{\Z}$ and using Lemma \ref{lem:h3} we obtain a natural exact sequence
\[
\xymatrix{
0\ar[r]
&\zhalf{\rbl{k}}\ar^-{\beta}[r]
&\dfrac{\ho{3}{\spl{2}{\laur{k}{t}}}{\zhalf{\Z}}}{\ho{3}{\spl{2}{k}}{\zhalf{\Z}}}\ar^-{\delta}[r]
&\ho{2}{k^\times}{\zhalf{\Z}}\ar[r]
&\ho{2}{\spl{2}{F}}{\zhalf{\Z}}
}
\]
Here the map $\beta$ is calculated as follows: Given $x\in \zhalf{\rbl{k}}$, choose $z\in 
\ho{3}{\spl{2}{k}}{\zhalf{\Z}}$ mapping to $x$. Then $\beta(x)$ is the  class of $\an{t}\cdot j(z)$
in 
\[
\frac{\ho{3}{\spl{2}{\laur{k}{t}}}{\zhalf{\Z}}}{\ho{3}{\spl{2}{k}}{\zhalf{\Z}}}.
\] 

\subsection{The map 
$\ho{3}{\spl{2}{\laur{k}{t}}}{\zhalf{\Z}}\to \zhalf{\qrpbker{k}}$}

Now let $\Delta:\ho{3}{\spl{2}{\laur{k}{t}}}{\Z}\to \qrpbker{k}$ denote the following 
composition of maps:
\[
\xymatrix{
\ho{3}{\spl{2}{\laur{k}{t}}}{\Z}\ar[r]
&\ho{3}{\spl{2}{k(t)}}{\Z}\ar[r]
&\rbl{k(t)}\ar[r]
&\qrpbker{k(t)}\ar^-{\delta_t}[r]
&\qrpbker{k}.
}
\]

The composite map 
\[
\xymatrix{
\rpbker{k}\ar[r]
&\rpbker{k(t)}\ar^-{S_v}[r]
&\qrpbker{k}_F=\qrpbker{k}\oplus\an{t}\cdot\qrpbker{k}  
}
\]
is just the projection on the first factor (since $S_v(\gpb{a})=1\otimes \gpb{\bar{a}}$ when 
$a\in U_v$) and hence the composite
\[
\xymatrix{
\rpbker{k}\ar[r]
&\rpbker{k(t)}\ar^-{\delta_t}[r]
&\qrpb{k}
}
\]
is the zero map. 

From the commutativity of the diagram 
\[
\xymatrix{
&\ho{3}{\spl{2}{k}}{\Z}\ar_-{j}[dl]\ar[d]\ar[r]
&\rbl{k}\ar[d]\\
\ho{3}{\spl{2}{\laur{k}{t}}}{\Z}\ar[r]
&\ho{3}{\spl{2}{k(t)}}{\Z}\ar[r]
&\rbl{k(t)}
}
\]
it follows that $\Delta(j(z))=0$ for all $z\in \ho{3}{\spl{2}{k}}{\Z}$ and hence there is an 
induced map 
\[
\tilde{\Delta}:
\frac{\ho{3}{\spl{2}{\laur{k}{t}}}{{\Z}}}{\ho{3}{\spl{2}{k}}{{\Z}}}\to 
\qrpbker{k}.
\]
We will show that, at least after tensoring with $\zhalf{\Z}$, $\tilde{\Delta}$ is an isomorphism.

We recall that the map 
\[
\xymatrix{
k^\times\wedge k^\times\cong\ho{2}{k^\times}{\Z}\ar^-{\iota}[r]
&\ho{2}{\spl{2}{k}}{\Z}\cong\mwk{2}{k}\\
}
\]
factors as 
\[
\xymatrix{
k^\times\wedge k^\times \ar^-{\tilde{\sigma}}[r]
&2\cdot\milk{2}{k}\ar@{^{(}->}[r]
&\mwk{2}{k}
}
\]
where $\tilde{\sigma}(a\wedge b)=2\{ a\}\{ b\}$. In particular, 
\[
\ker{\iota}=\ker{\tilde{\sigma}}=
\image{\delta:\ho{3}{\spl{2}{\laur{k}{t}}}{\Z}\to \ho{2}{k^\times}{\Z}}.
\]
(In fact, this is, of course,
 the subgroup of $k^\times\wedge k^\times$ generated by terms of the form $(1-a)\wedge a$, 
where $a\not= 0,1$.)
 
The square 
\[
\xymatrix{
k^\times \wedge k^\times\ar^-{\tilde{\sigma}}[r]\ar_-{\rho}[d]
&2\cdot\milk{2}{k}\ar@{^{(}->}[d]\\
\asym{2}{\Z}{k^\times}\ar^-{\sigma}[r]
&\milk{2}{k}
}
\]
-- where $\sigma$ is the symbol map $a\asymm b\mapsto \{ a\}\{ b\}$ -- commutes, and the vertical maps 
become isomorphisms upon tensoring with $\zhalf{\Z}$.

Thus $\rho$ induces an isomorphism 
\[
\zhalf{\ker{\tilde{\sigma}}}\cong \zhalf{\ker{{\sigma}}}.
\]

We have a diagram with exact rows
\[
\xymatrix{
0\ar[r]
&\zhalf{\rbl{k}}\ar^-{\mathrm{Id}}[d]\ar^-{\beta}[r]
&\dfrac{\ho{3}{\spl{2}{\laur{k}{t}}}{\zhalf{\Z}}}{\ho{3}{\spl{2}{k}}{\zhalf{\Z}}}
\ar^-{\delta}[r]\ar^-{\tilde{\Delta}}[d]
&\zhalf{\ker{\tilde{\sigma}}}\ar[r]\ar^-{\rho}[d]
&0\\
0\ar[r]
&\zhalf{\rbl{k}}\ar[r]
&\zhalf{\rpbker{k}}\ar^-{\lambda_2}[r]
&\zhalf{\ker{\sigma}}\ar[r]
&0.
}
\]

To prove that $\tilde{\Delta}$ is an isomorphism, it is enough to show that both squares in this 
diagram commute. 

Let $x\in \zhalf{\rbl{k}}$. Let $z\in \ho{3}{\spl{2}{k}}{\zhalf{Z}}$ map to $x$. Then $\beta(x)$
is represented by the element $\an{t}\cdot j(z)\in \ho{3}{\spl{2}{\laur{k}{t}}}{\zhalf{\Z}}$. 
This element maps to $\an{t}\cdot x \in \zhalf{\rbl{k(t)}}$ under the composite 
\[
\ho{3}{\spl{2}{\laur{k}{t}}}{\zhalf{\Z}}\to \ho{3}{\spl{2}{k(t)}}{\zhalf{\Z}}\to
\zhalf{\rbl{k(t)}}.
\]
 This element, in turn, maps to $x\in \zhalf{\rbl{k}}\subset\zhalf{\rbl{k(t)}}$ under the map 
$\delta_t$.  This shows that the left-hand square commutes.  

Finally, the commutativity of the right-hand square follows from
\begin{lem}\label{lem:h3comm}
For any infinite field $k$, the diagram
\[
\xymatrix{
\ho{3}{\spl{2}{\laur{k}{t}}}{\zhalf{\Z}}\ar^-{\delta}[r]\ar^-{\Delta}[d]
&\zhalf{(k^\times\wedge k^\times)}\ar^-{\rho}[d]\\
\zhalf{\rpbker{k}}\ar^-{\lambda_2}[r]
&\zhalf{\asym{2}{\Z}{k^\times}}
}
\]
commutes.
\end{lem}
\begin{proof} 

Recall that the Mayer-Vietoris exact sequence associated to the amalgamated product decomposition 
\[
G=\spl{2}{\laur{k}{t}}= \spl{2}{k[t]}\star_{\Gamma}\spl{2}{k[t]}^{A(t)}=G_1\star_{\Gamma}G_2
\] 
is the long exact homology sequence asspciated to the short exact sequence of 
$\Z[G]$-modules
\[
0\to \Z[G/\Gamma]\to \Z[G/G_1]\oplus \Z[G/G_2]\to \Z\to 0.
\]
This is therefore the long exact sequence resulting from the short exact sequence of 
complexes
\[
0\to T(G,\Z[G/\Gamma])\to T(G,\Z[G/G_1])\oplus T(G,\Z[G/G_2])\to T(G,\Z)\to 0
\]
(where $G$ acts on the complex $L_\bullet$ via the map $G\to \spl{2}{k}$ 
obtained by sending $t$ to $1$).

Suppose given 
$w\in \ker{\iota:\ho{2}{k^\times}{\zhalf{\Z}}\to \ho{2}{G_1}{\zhalf{\Z}}}=\ker{\tilde{\sigma}}$. 
Suppose further 
that $w$ is represented by the cycle $z\otimes 1\in B_2(k^\times)\otimes_{\Z[k^\times]}\zhalf{\Z}$.

Let $\tilde{z}=(0,0,z\otimes (\infty))$ be the corresponding element of 
\[
T_2(k^\times,\zhalf{\Z})= \bigoplus_{p+q=2}D_{p,q}(k^\times,\zhalf{\Z})=
\bigoplus_{p=0}^2 \left( B_p(k^\times)\otimes_{Z[k^\times]} (\zhalf{L_{2-p}})\right).
\]

Since $w$ maps to $0$ in $\ho{2}{G_1}{\zhalf{\Z}}$, it follows that the image of 
$\tilde{z}$ in $T_2(G_1,\zhalf{\Z})$ is a boundary. Thus there exists 
$x=(x_0,x_1,x_2,x_3)\in T_3(G_1,\zhalf{\Z})$ satisfying $dx=\tilde{z}$.

Since the first two components of $\tilde{z}$ are zeroes, it follows that $x_0$ represents an 
element, $\alpha$,  of $E^3_{0,3}=\zhalf{\rpbker{k}}$ and that   
$d^3(\alpha)\in E^3_{2,0}\cong\ho{2}{k^\times}{\zhalf{\Z}}$ 
is represented by $(\id{B_2(G_1)}\otimes d)(x_2)\in B_2(G_1)\otimes \zhalf{L_0}$. However, 
$(\id{B_2}\otimes d)(x_2)\pm (d\otimes \id{L_0})(x_3)=z\otimes (\infty)$. Hence, $d^3(\alpha)=w$ in  
$\ho{2}{k^\times}{\zhalf{\Z}}\cong \ho{2}{G_1}{\zhalf{L_0}}$. 

Let $C_t:G_1\to G_2$ denote the isomorphism given by conjugating by $A(t)$: 
$C_t(g)=A(t)^{-1}gA(t)=g^{A(t)}$. Observe that $C_t$ induces the identity map of the diagonal 
subgroup $k^\times$, since $A(t)$ commutes with other diagonal matrices. 

Thus $C_t(x)=(C_t(x_0),C_t(x_1),C_t(x_2), C_t(x_3))\in T_3(G_2,\zhalf{\Z})$. Then $C_t(x_0)$ represents
an element, $C_t(\alpha)$, of $E^3_{0,3}=\zhalf{C_t(\rpbker{k})}$ and $d^3(C_t(\alpha))=w$ in $E^3_{2,0}=  
\ho{2}{k^\times}{\zhalf{\Z}}$.

Now the cycle $\tilde{z}\in T_2(k^\times,\zhalf{\Z})$ maps  to the cycle 
$\tilde{z}\otimes \Gamma\in T_2(G,\zhalf{\Z[G/\Gamma]})$, and this cycle in turn represents the 
class $w\in \ho{2}{k^\times}{\zhalf{\Z}}\cong \ho{2}{\Gamma}{\zhalf{\Z}}$. 

Under the map 
\[
T_2(G,\zhalf{\Z[G/\Gamma]})\to T_2(G,\zhalf{\Z[G/G_1]})\oplus T_2(G,\zhalf{\Z[G/G_2]})
\] 
$\tilde{z}\otimes \Gamma$ maps to $(\tilde{z}\otimes G_1,\tilde{z}\otimes G_2)$. 
By the calculations above, this is the boundary of 
\[
(x\otimes G_1,C_t(x)\otimes G_2)\in T_3(G,\zhalf{\Z[G/G_1]})\oplus T_3(G,\zhalf{\Z[G/G_2]}).
\] 
This element in turn maps to the cycle $W= C_t(x)-x\in T_3(G,\zhalf{\Z})$ under the map
\[
T_3(G,\zhalf{\Z[G/G_1]})\oplus T_3(G,\zhalf{\Z[G/G_2]})\to T_3(G,\zhalf{\Z}). 
\]

By construction, the cycle $W$ represents an element of $\ho{3}{G}{\zhalf{\Z}}$ which maps 
to $w\in \ho{2}{k^\times}{\zhalf{\Z}}$ under the connecting homomorphism of the Mayer-Vietoris 
sequence of the amalgamated product. 

Now the image of $W$ in $T_3(\spl{2}{k(t)},\zhalf{\Z})$ represents a homology class in 
$\ho{3}{\spl{2}{k(t)}}{\zhalf{\Z}}$ which maps to $\an{t}\alpha-\alpha\in \zhalf{\rbl{k(t)}}$. 
By definition, this maps to the class $\alpha\in \zhalf{\qrpbker{k}}$ under the composite 
\[
\xymatrix{
\rbl{k(t)}\ar^-{S_t}[r]
&\qrpbker{k}_{k(t)}\ar^{\rho_t}[r]
&\qrpbker{k}.
}
\]
Thus $\Delta(W)=\alpha$ and hence 
\[
\lambda_2(\Delta(W))= \lambda_2(\alpha)=\rho(d^3(\alpha))=\rho(w)=\rho(\delta(W))
\]
as required. 
\end{proof}
\section{Some Examples and Special Cases}
The main result of section \ref{sec:h3} is that 
\[
\ho{3}{\spl{2}{\laur{k}{t}}}{\zhalf{\Z}}\cong \ho{3}{\spl{2}{k}}{\zhalf{\Z}}\oplus \zhalf{\rpbker{k}}.
\]

By Theorem \ref{thm:rblker} (2) and Lemma \ref{lem:rpbker} 
 there are natural short exact sequences 
\[
0\to \zhalf{\rblker{k}}\to \ho{3}{\spl{2}{k}}{\zhalf{\Z}}\to\zhalf{\kind{k}}\to 0
\]
and
\[
0\to \zhalf{\rblker{k}}\to \zhalf{\rpbker{k}}\to \zhalf{\pb{k}}\to 0. 
\]

Thus we have:
\begin{prop}\label{prop:h3sl2}
Let $k$ be an infinite field. Then there is a natural surjective homomorphism 
\[
\ho{3}{\spl{2}{\laur{k}{t}}}{\zhalf{\Z}}\to \zhalf{\kind{k}}\oplus\zhalf{\pb{k}}
\]
whose kernel is isomorphism to $\rblker{k}^{\oplus 2}$.
\end{prop}


By Theorem \ref{thm:rblker} (4), we deduce:

\begin{cor}
Let $k$ be a quadratically closed or real-closed field. Then there is a natural isomorphism
\[
\ho{3}{\spl{2}{\laur{k}{t}}}{\zhalf{\Z}}\cong \zhalf{\kind{k}}\oplus\zhalf{\pb{k}}.
\]
\end{cor}




By Theorem \ref{thm:rblker} (5) we have:

\begin{cor}\label{cor:loc}
Let $k$ be a local field with finite residue field $\bar{k}$ of odd order. 
If $\Q_3\subset k$ we suppose 
that $[k:\Q_3]$ is odd. Then there is a natural (split) short exact sequence of the form
\[
0\to \zhalf{\pb{\bar{k}}}^{\oplus 2} \to 
\ho{3}{\spl{2}{\laur{k}{t}}}{\zhalf{\Z}}\to  \zhalf{\kind{k}}\oplus\zhalf{\pb{k}}\to 0. 
\]
\end{cor}

\begin{exa} In particular, if $p\geq 3$ is prime  there is a natural decomposition
\[
\ho{3}{\spl{2}{\laur{\Q_p}{t}}}{\zhalf{\Z}}\cong \zhalf{\kind{\Q_p}}\oplus \zhalf{\pb{\Q_p}}\oplus
{\zhalf{\pb{\F{p}}}}^{\oplus 2}.
\] 
\end{exa}

We consider the case $k=\Q$. By Theorem \ref{thm:rblker} (6) there is  a surjective map 
\[
\xymatrix{
\zhalf{\rblker{\Q}}\ar@{->>}[r]
& \bigoplus_{p\mbox{ \tiny prime}}\zhalf{\pb{\F{p}}}\cong \bigoplus_{p}\Z/(p+1)'.
}
\] 
It is an open question whether this  map is an isomorphism. We note, however, that Theorem 6.1 of 
\cite{knud:laurent} implies that $\rblker{\Q}$ is a torsion group.  

Furthermore, we have $\kind{\Q}\cong \Z/24$ and $\bl{\Q}\cong \Z/6$ so that 
\[
\zhalf{\kind{\Q}}\cong\zhalf{\bl{\Q}}\cong \Z/3.
\]

From the exact sequence
\[
0\to \bl{\Q}\to \pb{\Q}\to \asym{2}{\Z}{\Q^\times}\to K_2(\Q)\to 0 
\] 
and the fact that $K_2(\Q)$ is a torsion group, we deduce that 
\[
\zhalf{\pb{\Q}}\cong \Z/3\oplus V
\]
where $V$ is a free $\zhalf{\Z}$-module of countably infinite rank.  Thus by Proposition 
\ref{prop:h3sl2} there is a short exact sequence 
\[
0\to \zhalf{\rblker{\Q}}^{\oplus 2}\to \ho{3}{\spl{2}{\laur{\Q}{t}}}{\zhalf{\Z}}\to
\Z/3\oplus \Z/3\oplus V\to 0. 
\]

Finally, we use the results of section \ref{sec:h3} to calculate the kernel and cokernel 
of the stabilization map
from $\ho{3}{\spl{2}{\laur{k}{t}}}{\zhalf{\Z}}$ to $\ho{3}{\mathrm{SL}(\laur{k}{t})}{\zhalf{\Z}}$:

\begin{thm}Let $k$ be an infinite field.
The cokernel of the map 
\[
\ho{3}{\spl{2}{\laur{k}{t}}}{\zhalf{\Z}} \to \ho{3}{\mathrm{SL}(\laur{k}{t})}{\zhalf{\Z}}
\]
is naturally isomorphic to $\zhalf{\milk{3}{k}}\oplus \zhalf{\milk{2}{k}}$.

The kernel of this map is naturally isomorphic to $\zhalf{\rblker{k}}\oplus\zhalf{\rpbker{k}}$.
\end{thm} 

\begin{proof}
For any ring $A$, Suslin has shown (\cite{sus:bloch}) that the Hurewicz homomorphism 
$K_3(A)\to \hoz{3}{\mathrm{GL}(A)}$ induces an isomorphism
\[
\frac{K_3(A)}{\{ -1\}\cdot K_2(A)}\cong \hoz{3}{\mathrm{SL}(A)}.
\]

Since the element $\{ -1\}\in K_1(A)$ has order $2$, it follows that there is an induced isomorphism
$\zhalf{K_3(A)}\cong \ho{3}{\mathrm{SL}(A)}{\zhalf{\Z}}$. 

In particular, for an infinite field $k$ 
\begin{eqnarray*}
\ho{3}{\mathrm{SL}{\laur{k}{t}}}{\zhalf{\Z}}&\cong & \zhalf{K_3(\laur{k}{t})}\\
\cong  \zhalf{K_3(k)}\oplus \zhalf{K_2(k)}
&\cong & \ho{3}{\mathrm{SL}(k)}{\zhalf{\Z}}\oplus \ho{2}{\mathrm{SL}(k)}{\zhalf{\Z}}.\\
\end{eqnarray*}

Thus the stabilization map induces an map of short exact sequences with compatible splittings 
\[
\xymatrix{
0\ar[r]
&\ho{3}{\spl{2}{k}}{\zhalf{\Z}}\ar[r]\ar[d]
&\ho{3}{\spl{2}{\laur{k}{t}}}{\zhalf{\Z}}\ar[r]\ar[d]
&\zhalf{\rpbker{k}}\ar[r]\ar[d]
&0\\
0\ar[r]
&\ho{3}{\mathrm{SL}(k)}{\zhalf{\Z}}\ar[r]
&\ho{3}{\mathrm{SL}(\laur{k}{t})}{\zhalf{\Z}}\ar[r]
&\zhalf{K_2(k)}\ar[r]
&0\\
}
\]

The kernel of the stabilization map $\ho{3}{\spl{2}{k}}{\zhalf{\Z}}\to \ho{3}{\mathrm{SL}(k)}{\zhalf{\Z}}$ 
is isomorphic to $\zhalf{\rblker{k}}$  and the cokernel is isomorphic 
to $\zhalf{\milk{3}{k}}$. 

We conclude by showing that 
 the map $\zhalf{\rpbker{k}}\to \zhalf{K_2(k)}$ in the diagram above is the zero map:

Recall that the map $\Delta:\ho{3}{\spl{2}{\laur{k}{t}}}{\zhalf{\Z}}\to \zhalf{\rpbker{k}}$ 
factors as 
\[
\xymatrix{
\ho{3}{\spl{2}{\laur{k}{t}}}{\zhalf{\Z}}\to \ho{3}{\spl{2}{k(t)}}{\zhalf{\Z}}\to
\zhalf{\rbl{k(t)}}\ar^-{\delta_t}[r]
&\zhalf{\rpbker{k}}
}
\]
and there is a compatible factorization of the map $\ho{3}{\mathrm{SL}(\laur{k}{t})}{\zhalf{\Z}}
\to K_2(k)$:
\[
\xymatrix{
\ho{3}{\mathrm{SL}(\laur{k}{t})}{\zhalf{\Z}}\to \ho{3}{\mathrm{SL}(k(t))}{\zhalf{\Z}}\cong
\zhalf{K_3(k(t))}\ar^-{\delta_t}[r]
&\zhalf{K_2(k)}.
}
\]

Let  $\alpha\in \zhalf{\rpbker{k}}$. Consider the element 
$\pf{t}\cdot \alpha\in \zhalf{\rpbker{(k(t)}}$. Note that $\delta_t(\pf{t}\cdot\alpha)=\alpha$.

However 
\[
\pf{t}\cdot \alpha \in \zhalf{\rbl{k(t)}}=\ker{\zhalf{\rpbker{k(t)}}\to \asym{2}{\Z}{k(t)^\times}}    
\]
since $\an{t}$ acts trivially on $\asym{2}{\Z}{k(t)}^\times$. 

Likewise
\[
\pf{t}\cdot \alpha \in \zhalf{\rblker{k(t)}}=\ker{\zhalf{\rbl{k(t)}}\to \zhalf{\pb{k(t)}}}    
\]
since $\an{t}$ acts trivially on $\pb{k(t)}$. 

But 
\[
\zhalf{\rblker{k(t)}}\cong \ker{\ho{3}{\spl{2}{k(t)}}{\zhalf{\Z}}\to 
\ho{3}{\mathrm{SL}(k(t))}{\zhalf{\Z}}}
\]

It follows that if $\beta\in \ho{3}{\spl{2}{k(t)}}{\zhalf{\Z}}$ is any lifting of 
$\pf{t}\cdot\alpha$, then $\beta$ maps to $0$ in $\ho{3}{\mathrm{SL}(k(t))}{\zhalf{\Z}}$. 
Therefore, $\alpha$ maps to $0$ in $\zhalf{K_2(k)}$ as claimed.
\end{proof}

\begin{cor} Let $k$ be a quadratically closed or a real-closed field. Then
the kernel of the stabilization map 
\[
\ho{3}{\spl{2}{\laur{k}{t}}}{\zhalf{\Z}} \to \ho{3}{\mathrm{SL}(\laur{k}{t})}{\zhalf{\Z}}
\]
is naturally isomorphic to $\zhalf{\pb{k}}$.
\end{cor}

\bibliography{Laurent}
\end{document}

%% file: Laurent-command.tex
\newcommand{\Q}{\Bbb{Q}}
\newcommand{\F}[1]{\Bbb{F}_{#1}}

\newcommand{\Z}{\Bbb{Z}}

\newcommand{\R}{\Bbb{R}}
\newcommand{\C}{\Bbb{C}}
\newcommand{\K}{\mathcal{K}}

\newcommand{\spl}[2]{\mathrm{SL}_{#1}(#2)}

\newcommand{\gl}[2]{\mathrm{GL}_{#1}(#2)}

\newcommand{\pb}[1]{\mathcal{P}(#1)}

\newcommand{\rpb}[1]{{\mathcal{RP}}(#1)}
\newcommand{\rpbker}[1]{\mathcal{RP}_1(#1)}
\newcommand{\qrpb}[1]{\widetilde{\mathcal{RP}}(#1)}

\newcommand{\qrpbker}[1]{\widetilde{\mathcal{RP}}_1(#1)}

\newcommand{\bl}[1]{\mathcal{B}(#1)}

\newcommand{\rbl}[1]{{\mathcal{RB}}(#1)}
\newcommand{\rblker}[1]{\mathcal{RB}_0(#1)}%
\newcommand{\gpb}[1]{\left[ #1\right]} 
\newcommand{\suss}[2]{\psi_{#1}\left( #2\right)}

\newcommand{\ks}[2]{{\K}^{\mbox{\tiny $(#1)$}}_{#2}}

\newcommand{\pn}[2]{\mathbb{P}^{#1}(#2)}
\newcommand{\projl}[1]{\pn{1}{#1}}

\newcommand{\laur}[2]{{#1}[{#2},{#2}^{-1}]}

\newcommand{\shift}[1]{{#1}_{-1}}

\newcommand{\<}{\langle}
\renewcommand{\>}{\rangle}
\newcommand{\il}[1]{\< #1 \>}

\newcommand{\id}[1]{\mathrm{Id}_{#1}}

\renewcommand{\ker}[1]{\mathrm{Ker}(#1)}
\newcommand{\image}[1]{\mathrm{Im}(#1)}

\newcommand{\ab}[1]{#1^{\mbox{\tiny ab}}}

\newcommand{\sgr}[1]{\Z[{#1}^\times/({#1}^\times)^2]}
\newcommand{\an}[1]{\left\langle{#1}\right\rangle}
\newcommand{\pf}[1]{\left\langle\!\left\langle{#1}\right\rangle\!\right\rangle}
\newcommand{\aug}[1]{\mathcal{I}_{#1}}

\newcommand{\laurs}[2]{{#1}\left(\!\left( #2 \right)\!\right)}

\newcommand{\matr}[4]{\left[\begin{array}{cc}
#1&#2\\
#3&#4\\
\end{array}
\right]}

\newcommand{\extpow}[1]{{#1}\wedge{#1}}

\newcommand{\asymm}{\circ}
\newcommand{\asym}[3]{\mathrm{S}^{#1}_{#2}(#3)}

\newcommand{\zhalf}[1]{{#1}\left[ \tfrac{1}{2}\right]}

\newcommand{\sgn}[1]{\mathrm{sgn}(#1)}

\newcommand{\mwk}[2]{K^{\mathrm{\small MW}}_{#1}({#2})}

\newcommand{\milk}[2]{K^{\mathrm{\small M}}_{#1}({#2})}

\newcommand{\kind}[1]{K^{\mathrm{\small ind}}_3(#1)}

\newcommand{\hgen}[2]{{[#1,#2]}}

\newcommand{\ho}[3]{\mathrm{H}_{#1}(#2,#3 )}
\newcommand{\hoz}[2]{\ho{#1}{#2}{\Z}}

